\documentclass[12pt,reqno]{amsart}
\usepackage{amsmath,amssymb,amscd,latexsym,amsthm,mathrsfs}
\usepackage{color}
\usepackage[unicode]{hyperref}
\usepackage{hypbmsec,longtable}
\ifx\pdfoutput\undefined\else
\hypersetup{
colorlinks=true,
linkcolor=blue,
citecolor=blue,
urlcolor=blue,
filecolor=blue,
bookmarksnumbered=true,
pdfstartview=FitH,
pdfhighlight=/N
}
\fi
\newtheorem{theorem}{Theorem}

\theoremstyle{remark}
\newtheorem{remark}{Remark}
\newtheorem{example}{Example}
\newtheorem*{acknowledgements}{Acknowledgements}
\renewcommand{\d}{{\mathrm d}}

\newcommand{\qede}{\hspace*{\fill}$\Diamond$\medskip}

\hypersetup{pdfauthor={Rogers, Wan, Zucker},pdftitle={Moments of elliptic integrals and critical $L$-values}, bookmarks=true, pdfstartview={FitV}}

\begin{document}

\title{Moments of elliptic integrals and critical $L$-values}

\author{M. Rogers}
\address{Department of Mathematics and Statistics, Universit\'e de Montr\'eal,
CP 6128 succ.\ Centre-ville, Montr\'eal Qu\'ebec H3C\,3J7, Canada}
\email{mathewrogers@gmail.com}

\author{J. G. Wan}
\address{School of Mathematical and Physical Sciences, The University of Newcastle, Callaghan, NSW 2308, Australia}
\email{james.wan@newcastle.edu.au}
\address{Pillar of Engineering Systems and Design, Singapore University of Technology and Design, 20 Dover Drive Singapore 138682}
\email{james\_wan@sutd.edu.sg}

\author{I. J. Zucker}
\address{Department of Physics, Kings College London, Strand, London WC2R 2LS, UK}
\email{jzucker@btinternet.com}

\date{April 9, 2013}


\begin{abstract}
We compute the critical $L$-values of some weight 3, 4, or 5 modular forms, by transforming them into integrals of the complete elliptic integral $K$.   In doing so, we prove closed form formulas for some moments of $K'^3$.
Many of our $L$-values can be expressed in terms of Gamma functions, and this also gives new lattice sum evaluations.
\end{abstract}

\maketitle
%==================================================

\section{Introduction}
\label{s-intro}

Let $K$ and $K'$ denote the complete elliptic integrals of the first kind, defined as follows:
\begin{align*}
K(k):=&\int_{0}^{1}\frac{\d u}{\sqrt{(1-u^2)(1-k^2u^2)}} =\frac{\pi}{2}\, {_2F_1}\left({{\frac12,\frac12}\atop 1};k^2\right), \\
K'(k):=&K(k'), \quad k' := \sqrt{1-k^2}.
\end{align*}
One of the present authors (Wan) studied integrals of elliptic integrals, or ``moments of elliptic integrals" in \cite{wan}.  That paper includes several conjectures which have since been settled. The last remaining conjecture is \begin{equation}\label{James formula}
\int_{0}^{1}K'(k)^3 \d k=\frac{\Gamma^8\left(\frac14\right)}{128\pi^2}.
\end{equation}
%The question mark above the equals sign denotes numerical equality.
There are several versions of \eqref{James formula}. For instance, we can reformulate the integral using transformations for $K'$.  We can also use integration by parts, because derivatives of elliptic integrals equal linear combinations of elliptic
integrals. Some examples found in \cite{wan} include
\begin{equation}
\int_0^1 K'(k)^3 \d k = \frac{10}{3} \int_0^1 K(k)^3 \d k = 5 \int_0^1 k K'(k)^3 \d k. \label{reformulations}
\end{equation}
In the first part of this paper, we prove formula \eqref{James formula}, and provide some intuition about how to discover related integrals. We also settle some additional integrals involving $K'^3$, and give closed form evaluations of integrals containing higher powers of $K'$ (see equations \eqref{k7} and \eqref{k11}). In the second part of the paper, we study a more general phenomenon, where critical $L$-values of \textit{odd weight} modular forms can be expressed in terms of Gamma values. Some new lattice sums are produced from our investigation.

We note that Y. Zhou, in a 2013 preprint \cite{zhou}, used methods based on spherical harmonics to prove both equations \eqref{James formula} and \eqref{equivconj}. Zhou also applied his ideas to many other integrals.

\section{Critical $L$-values}

One of the main goals of this paper is to connect integrals like \eqref{James formula} to critical $L$-values of modular forms. We say that a function $f(\tau)$ is a modular form of weight $k$ and level $N$, if it satisfies
$$f\left(\frac{a\tau+b}{c\tau+d}\right)=(c\tau+d)^k f(\tau),$$
whenever $\left({{a \ b}\atop{c \ d}}\right) \in SL_2(\mathbb{Z})$, and $c\equiv 0\mod N$. We also require that $f(\tau)$ be holomorphic in the upper half plane, and  that it vanish at the cusps (that is, it is a cusp form).  The most interesting modular forms are the Hecke eigenforms, whose Fourier series 
\begin{equation*}
f(\tau)=\sum_{n=1}^{\infty}a_n e^{2\pi i n \tau},
\end{equation*}
have multiplicative coefficients $a_n$. If we attach an $L$-series to $f(\tau)$,
\begin{equation}\label{L-series definition}
L(f,s)=\sum_{n=1}^{\infty}\frac{a_n}{n^s},
\end{equation}
then $L(f,s)$ has a meromorphic continuation to the complex plane.  We say that $L(f,j)$ is a \textit{critical} $L$-value if
$j\in\{1,2,\dots,k-1\}$.  These numbers typically hold arithmetic significance; some properties
of critical $L$-values are summarized in \cite{KZ}.  For instance, the Birch and Swinnerton-Dyer conjecture predicts the value of $L(f,1)$, whenever $f(\tau)$ is a weight 2 cusp form attached to an elliptic curve.  In Theorem \ref{Theorem : central L}, we prove that equation \eqref{James formula} is equivalent to an explicit formula for a critical $L$-value of a weight $5$ cusp form.  In particular, we prove
\begin{equation}\label{James reformulated}
30L(g,4)=\int_{0}^{1}K'(k)^3 \d k,
\end{equation}
where
\begin{equation*}
g(\tau):=\eta^4(\tau)
\eta^2(2\tau)\eta^4(4\tau),
\end{equation*}
and $\eta(\tau)$ is the usual Dedekind eta function.  We then use properties of $g(\tau)$ to prove equation \eqref{James formula} in Theorem \ref{Theorem : James proof}.

Formula \eqref{James reformulated} naturally suggests looking at critical $L$-values of different modular forms.  Martin has classified all of the possible multiplicative eta quotients \cite{Ma}.  While it is possible to consider a much larger class of modular forms than just multiplicative eta quotients, these are typically the easiest examples to work with.  Martin's list contains precisely two examples of weight $5$.  The first is $g(\tau)$ above, and the second example is
\begin{equation*}
h(\tau):=\frac{\eta^{38}(8\tau)}{\eta^{14}(4\tau)\eta^{14}(16\tau)}.
\end{equation*}
It is easy to prove that
\begin{equation*}
192L(h,4)=\int_{0}^{1}\frac{K'(k)^3}{\sqrt{k}(1-k^2)^{3/4}}\d
k,
\end{equation*}
and based on the previous example, we guess that $L(h,4)$ should also be related to Gamma values.  We discovered the following identity after a brief numerical search:
\begin{equation}\label{New James formula}
\int_{0}^{1}\frac{K'(k)^3}{\sqrt{k}(1-k^2)^{3/4}}\d k = \frac{3\Gamma^8(\frac14)}{32\sqrt{2}\pi^2}.
\end{equation}
We started by calculating the integral on the left to high numerical precision (denoted $I$), and then we used the PSLQ algorithm
to search for a linear dependencies in the set
\begin{equation*}
\left\{\log|I|,~\log\pi,~\log\Gamma(1/3),~\log\Gamma(1/4),\dots,~\log2,~\log3,~\log5,\dots\right\}.
\end{equation*}
Bailey and Borwein used PSLQ to discover many identities among moments of elliptic integrals -- far more than what is currently proven \cite{combat}.

%A crucial property which allows us to prove \eqref{James formula}, \eqref{New James formula}, and various additional %examples, is that the modular forms which appear are \textit{lacunary} cusp forms (that is, the density of the nonzero %coefficients is zero). 
The crucial property which allows us to prove \eqref{James formula} and \eqref{New James formula} is that the attached modular forms are also binary theta functions.  It is often possible to rewrite the $L$-functions of the modular forms as Hecke $L$-functions involving Grossencharacters, but this connection is not usually essential.
In Section \ref{Sec: weight 3 cases} we describe some additional integrals which 
arise from weight $3$ cusp forms. We think that it is noteworthy that there appears to be fewer interesting formulas for $L$-values attached to even weight cusp forms. We discuss this in the conclusion. 

\section{Proof of the conjectures}
\label{James conjecture}

We relate conjecture \eqref{James formula} to critical $L$-values in the following theorem. The proof is typical of the approach we use for subsequent integrals, so we spell out the details.

\begin{theorem}\label{Theorem : central L}
Let $g(\tau)=\eta^4(\tau)\eta^2(2\tau)\eta^4(4\tau)$. Formula
\eqref{James formula} is equivalent to:
\begin{equation}\label{James reformulated2}
L(g,4)=\frac{\Gamma^8\left(\frac14\right)}{3840 \pi^2}.
\end{equation}
\end{theorem}
\begin{proof}  The proof follows from Ramanujan-style manipulations.  Set $k=\sqrt{\alpha}$, and notice that
\begin{equation*}
\int_{0}^{1}K'(k)^3 \d
k=\frac{\pi^3}{16}\int_{0}^{1}{_2F_1}\biggl({{\frac{1}{2},\frac{1}{2}}\atop 1};1-\alpha\biggr)^3 \frac{\d
\alpha}{\sqrt{\alpha}}.
\end{equation*}
Now we make a change of variables.  Set
\begin{equation}\label{elliptic nome}
q=\exp\biggl(-\pi\frac{{_2F_1}\bigl(\substack{\frac{1}{2},\frac{1}{2}\\1};1-\alpha\bigr)}{{_2F_1}\bigl(\substack{\frac{1}{2},\frac{1}{2}\\1};\alpha\bigr)}\biggr),
\end{equation}
and notice that
\begin{equation*}
\d \alpha=\alpha(1-\alpha){_2F_1}\biggl({{\frac{1}{2},\frac{1}{2}}\atop 1};\alpha\biggr)^2\frac{\d q}{q}.
\end{equation*}
It is standard to show that $q\in(0,1)$ when $\alpha\in(0,1)$, and $q$ is monotone with respect to $\alpha$.  The integral becomes
\begin{equation*}
\int_{0}^{1}K'(k)^3 \d
k=-\frac{1}{16}\int_{0}^{1}\sqrt{\alpha}(1-\alpha){_2F_1}\biggl({{\frac{1}{2},\frac{1}{2}}\atop 1};\alpha\biggr)^5\log^3
q\frac{\d q}{q}.
\end{equation*}
Consider the Dedekind eta function with respect to $q$, where $q=e^{2\pi i \tau}$:
\begin{equation*}
\eta(q):=q^{1/24}\prod_{n=1}^{\infty}(1-q^n).
\end{equation*}
By \cite[p.~124, Entry~12]{Be3} we have
\begin{equation*}
\sqrt{\alpha}(1-\alpha){_2F_1}\biggl({{\frac{1}{2},\frac{1}{2}}\atop 1};\alpha\biggr)^5=4\frac{\eta^{14}(q^2)}{\eta^4(q^4)},
\end{equation*}
and the integral reduces to
\begin{equation*}
\int_{0}^{1}K'(k)^3 \d k=-\frac{1}{4}\int_{0}^{1}\frac{\eta^{14}(q^2)}{\eta^4(q^4)}\log^3 q\frac{\d q}{q}=-4\int_{0}^{1}\frac{\eta^{14}(q^4)}{\eta^4(q^8)}\log^3 q\frac{\d q}{q}.
\end{equation*}
From \cite[Entry $t_{8,12,48}$ and Entry~$t_{8,18,60a}$]{so2}, it is easy to prove that
\begin{equation*}
\frac{\eta^{14}(q^4)}{\eta^4(q^8)}=4\eta^4(q^2)\eta^2(q^4)\eta^4(q^8) + \eta^4(q)\eta^2(q^2)\eta^4(q^4).
\end{equation*}
Therefore the integral becomes
\begin{equation*}
\begin{split}
\int_{0}^{1}K'(k)^3 \d k=&-16\int_{0}^{1}\eta^4(q^2)\eta^2(q^4)\eta^4(q^8)\log^3 q\frac{\d q}{q}-4\int_{0}^{1} \eta^4(q)\eta^2(q^2)\eta^4(q^4)\log^3 q\frac{\d q}{q}\\
=&-5\int_{0}^{1}\eta^4(q)\eta^2(q^2)\eta^4(q^4)\log^3 q\frac{\d q}{q}.
\end{split}
\end{equation*}
Switching to the more traditional notation for $\eta$ in terms of $\tau$ , we have
\begin{equation*}
\int_{0}^{1}K'(k)^3 \d k=30L(g,4),
\end{equation*}
where $g(\tau)=\eta^4(\tau)\eta^2(2\tau)\eta^4(4\tau)$.  It follows
that \eqref{James formula} and \eqref{James reformulated2} are
equivalent.
\end{proof}
%Formula \eqref{James reformulated} is apparently the first recorded example
%involving the $L$-value of a cusp form of weight $5$.  
In order to prove \eqref{James reformulated2}, we require the fact that $g$ is a binary theta function.  Glaisher \cite{glaisher} showed that
\begin{equation}\label{Somos formula}
g(\tau)=\frac{1}{4}\sum_{(n,m)\in\mathbb{Z}^2}(n-i m)^4 q^{n^2+m^2},
\end{equation}
where as usual $q=e^{2\pi i \tau}$. 
%Notice that we can also express the right-hand side of the identity as 
%\begin{equation*}   
%g(\tau)=\frac{1}{4}\sum_{\subseteq\mathbb{Z}[i]}\chi(I)q^{N(I)}. 
%\end{equation*} 

%Somos used experimental methods to deduce that
%\begin{equation}\label{Somos formula}
%g(\tau)=\frac{1}{4}\sum_{(n,m)\in\mathbb{Z}^2}(n-i m)^4 q^{n^2+m^2},
%\end{equation}
%where as usual $q=e^{2\pi i \tau}$ \cite{so1}.  Formula \eqref{Somos formula}
%is true, because both sides of the equation are modular forms on
%$\Gamma_{1}(4)$, and their Fourier coefficients agree for sufficiently many terms.

\begin{theorem}\label{Theorem : James proof} Formula \eqref{James reformulated2} is true.
\end{theorem}
\begin{proof}  Integrating \eqref{Somos formula} leads to
\begin{equation*}
L(g,4)=\frac{1}{4}\sum_{(n,m)\ne(0,0)}\frac{(n-i
m)^4}{(n^2+m^2)^4}=\frac{1}{4}\sum_{(n,m)\ne(0,0)}\frac{1}{(n+i
m)^4}.
\end{equation*}
The Weierstrass invariant $g_2(\tau)$  can be defined by
\begin{equation*}
g_2(\tau):=60 \sum_{(n,m)\ne(0,0)}\frac{1}{(n+\tau m)^4},
\end{equation*}
and can be calculated using
\begin{equation} \label{g2calc}
g_2(\tau)=\frac{64}{3}(1-k^2+k^4)K^4(k),
\end{equation}
where $k, K$ and $\tau$ obey the classical relations $k=\theta_2^2(e^{\pi i \tau})/\theta_3^2(e^{\pi i \tau})$ and $K = \frac{\pi}2\theta_3^2(e^{\pi i \tau})$.  
In the language of Eisenstein series, $g_2(\tau)=120 \zeta(4) E_4(\tau)$.

Standard evaluations show that $k=1/\sqrt2$ when $\tau=i$, and thus
\begin{equation*}
L(g,4)=\frac{1}{240}g_2(i)=\frac{1}{15}K^4\biggl(\frac{1}{\sqrt2}\biggr).
\end{equation*}
Since $K(1/\sqrt2)=\Gamma^2(\frac14)/(4\sqrt\pi)$, we
obtain
\begin{equation*}
L(g,4)=\frac{\Gamma^8(\frac14)}{3840\pi^2},
\end{equation*}
completing the proof.
\end{proof}

\medskip

\begin{remark}
Wan used numerical experiments to observe that the moments of $K'^3$ and $K^2 K'$ are related by a rational factor \cite{wan}:
\begin{equation} \label{equivconj}
\int_0^1 K'(k)^3\mathrm dk = 3 \int_0^1 K(k)^2 K'(k)\mathrm dk.
\end{equation}
We sketch a proof of equation \eqref{equivconj} here.  We can use Tricomi's Fourier series \cite[Sec.~6]{wan}:
\[ K(\sin t) = \sum_{n=0}^\infty  \frac{\Gamma(n+\frac12)^2}{\Gamma(n+1)^2}  \sin((4n+1)t), \]  
to deduce that
\begin{equation}
\sum_{n=0}^\infty \frac{\Gamma^2(n+\frac12)}{\Gamma^2(n+1)}\int_0^{\pi/2} \cos((4n+2)t)\cos(t)F(\sin t)\mathrm dt = \int_0^1 \bigl(k' K'(k)-k K(k)\bigr)F(k)\mathrm dk, \label{trigk}
\end{equation}
whenever $F$ is selected so that summation and integration are interchangeable.
To derive \eqref{trigk}, perform the change of variables $k \mapsto \cos t$ on the right hand side, and then use the trigonometric identity 
\[ \cos(t)\cos((4n+1)t)-\sin(t)\sin((4n+1)t) = \cos((4n+2)t). \]
We then set $F(k) = K^2(k)/k'$ in \eqref{trigk}. The right-hand side simplifies under the transformation $k \mapsto k'$, and cancellations occur on the left-hand side due to orthogonality. After simplifying we have
\begin{equation*}
\frac{\pi^2}{8}\sum_{n=0}^\infty \frac{\Gamma^4(n+\frac12)}{\Gamma^4(n+1)} \, _4F_3\biggl({{\frac12,\frac12,-n,-n}\atop{1,\frac12-n,\frac12-n}};1\biggr)= \int_0^1 K'(k)^3-K(k)^2 K'(k)\mathrm dk.
\end{equation*}
The sum is precisely $2\int_0^1 K(k)^2 K'(k)\mathrm dk$, as established near the end of \cite{wan}.  This proves \eqref{equivconj}. 

Closed forms such as \eqref{James formula} and \eqref{equivconj}, together with the use of Legendre's relation in \cite{wan}, produce additional evaluations involving the complete elliptic integral of the second kind $E$, for example
\[ \int_0^1 E(k)K'(k)^2 \mathrm dk = \frac{\pi^3}{12}+\frac{\Gamma^8(\frac14)}{384\pi^2}. \]
 \qede
\end{remark}

Before proving \eqref{New James formula}, we note that there are various reformulations of the integral.  For example, a quadratic transformation gives
\[ \int_{0}^{1}\frac{K'(k)^3}{\sqrt{k}(1-k^2)^{3/4}}\d k =\frac{1}{8\sqrt{2}} \int_0^1 \frac{(1+k)^3K'(k)^3}{k^{3/4}\sqrt{1-k}}\d k. \]
Experimentally, we discover the binary theta expansion
\begin{align} \label{hproduct} 
h(\tau) & = \prod_{n=1}^\infty \frac{(1-(-1)^n q^{4n})^{14}}{(1-q^{8n})^4} \\
 & = \frac12 \sum_{m, n = -\infty}^\infty (-1)^m (2n+1-2i m)^4 \, q^{(2m)^2+(2n+1)^2}.  \label{hseries}
\end{align}
The proof of \eqref{hseries} is slightly tedious.  We may show that both sides are modular forms on $\Gamma_1(64)$, that their Fourier coefficients agree for sufficiently many terms, and then appeal to the valence formula. Alternatively, the double sum can be built up from formulas of derivatives of $\theta_2(q^4)$ and $\theta_4(q^4)$ (Ramanujan's $\varphi(-q^4)$ and $\psi(q^8)$) in \cite[pp.~122--123]{Be3}. The result can be written in terms of $k$ and $K$, and then reduced to the product \eqref{hproduct}.

\begin{theorem}
Formula \eqref{New James formula} is true.
\end{theorem}
\begin{proof}
From \eqref{hseries}, we have the (lattice) sum
\[ 2L(h,4)= \sum_{m, n = -\infty}^\infty \frac{(-1)^m(2n+1-2im)^4}{\bigl((2n+1)^2+(2m)^2\bigr)^4} = \sum_{m, n = -\infty}^\infty \frac{(-1)^m}{(2n+1+2im)^4}, \]
so we are reduced to showing 
\begin{equation} \label{heasy}
\sum_{m, n = -\infty}^\infty \frac{(-1)^m}{(2n+1+2im)^4} = \frac{\Gamma^8(\frac14)}{1024\sqrt2 \pi^2}.
\end{equation}
This can be achieved using routine manipulations on $g_2(\tau)$. By \eqref{g2calc}, we have
\[ \sum_{(m,n) \ne (0,0)} \frac{(-1)^m}{(n+m\tau)^4} = \frac{1}{60}\bigl(2g_2(2\tau)-g_2(\tau)\bigr). \]
This leads to
\begin{align} \nonumber
\sum_{m,n} \frac{(-1)^m}{(2n+1+m\tau)^4} & = \frac{1}{60}\bigl(2g_2(2\tau)-g_2(\tau)\bigr) - \frac{1}{2^4}\frac{1}{60}\bigl(2g_2(\tau)-g_2(\tau/2)\bigr) \\
& = \frac{1}{960}\bigl(g_2(\tau/2)-18g_2(\tau)+32g_2(2\tau)\bigr). \label{g2final}
\end{align}
Now \eqref{heasy} follows by setting $\tau=2i$ in \eqref{g2final} and simplifying. The required values of $k$ (which are the 1st, 4th and 16th singular values) and $K$ at these values can be found in \cite{Be3} and \cite{Zu}.
\end{proof}

\begin{example}
We may look at other critical $L$-values of $g$ and $h$. Simple calculations give
\begin{align*}
L(g,3) & = \frac{1}{2\pi} \int_0^1 k K'(k)^2 K(k)\d k, \\
L(g,2) & = \frac{1}{\pi^2} \int_0^1 k K(k)^2 K'(k)\d k, \\
L(g,1) & = \frac{1}{\pi^3} \int_0^1 k K(k)^3\d k. \\
\end{align*}
One may notice the similarity between $L(g,3)$ and $L(g,2)$; indeed they are related by the substitution $k \mapsto k'$. This similarity can be explained, since $g$ is a weight 5 modular form, so there is a \textit{functional equation} connecting $L(g, s)$ and $L(g,5-s)$ (with some Gamma factors). In view of \eqref{equivconj}, \eqref{reformulations} and \cite[Thm.~5]{wan}, the values $\pi^{5-s} L(g,s)$ where $s \in \{1,2,3,4\}$ are all related to each other by rational constants.

Similarly, we have $L(h,1) = 384/\pi^3 \, L(h,4)$, which can be explained by either using the functional equation or a substitution in the integral. Experimentally we also observe that $L(h,4) = \frac{\pi}{4} L(h,3)$; an integral formulation for this is
\begin{equation} \label{old13}
 \int_{0}^{1}\frac{K'(k)^2K(k)}{\sqrt{k}(1-k^2)^{3/4}}\d k = \frac{\Gamma^8(\frac14)}{32\sqrt{2}\pi^2},
\end{equation}
and an equivalent formulation as a sum is
\[ \sum_{m, n = -\infty}^\infty \frac{(-1)^m(2n+1-2mi)}{(2n+1+2im)^3} = \frac{\Gamma^8(\frac14)}{256\sqrt2 \pi^2}. \] 
We are very grateful to Y.~Zhou, who has  kindly supplied us with a proof of \eqref{old13} (plus some generalizations) using techniques from \cite{zhou}; since such techniques differ significantly from the ones in the current paper, we do not include the proof here. Our approach to \eqref{old13} is modular (using some ideas from \cite{RZ2}), but is more complicated, and we will present it in a future paper.\footnote{M.~Rogers, J.\,G.~Wan and I.\,J.~Zucker, \emph{Integrals from Lattice Sums}, work in progress, 2013. \label{footnote1}} In general, we expect the critical $L$-values of odd weight modular forms to be related by algebraic constants and powers of $\pi$, though the computation of these constants is not trivial -- see the conclusion for discussion.
\qede
\end{example}

\begin{remark}
Generalizations of \eqref{James formula} are possible. Starting with the level 4, weight 9 newform
\[ f(q) = \frac14 \sum_{m,n} (m-in)^8 q^{m^2+n^2}, \]
then with $k=\theta_2^2(q)/\theta_3^2(q)$ as usual, it can be checked by computing the derivatives of $\theta_3(q) = \sum_n q^{n^2}$ that
\begin{equation}\label{e8ask} f(q) =  \frac{8(4k^2k'^2+k^4k'^4)}{\pi^9}K^9(k). \end{equation}
We have the following $L$-value,
\[ L(f,8) = \frac14 \sum_{(n,m)\ne (0,0)} \frac{1}{(m+in)^8} = \frac{\Gamma^{16}(\frac14)}{2^{10}\,525\,\pi^4}. \]
The last equality holds, since as an Eisenstein series the sum equals $\frac12 \zeta(8) E_8(i)$ and $E_8=E_4^2$. Writing $L(f,8)$ as an integral using \eqref{e8ask}, we obtain
\begin{equation} \label{k7}
\int_0^1 k(4+k^2-k^4)K'(k)^7\mathrm{d}k = \frac{3\, \Gamma^{16}(\frac14)}{2^{12}\,5\,\pi^4}.
\end{equation}
Experimentally, $L(f,8-i)/\pi^i$ are all related by rational constants. Since 
\[ \sum_{(n,m)\ne (0,0)} \frac{1}{(m+in)^{4k}}\]
can always be expressed as a rational number times a power of $K(1/\sqrt2)$ \cite{hur}, we have found  generalizations of the above result which involve higher powers of $K'$, for instance
\begin{equation}  \label{k11}
\int_0^1 k(16-92k^2+93k^4-2k^6+k^8)K'(k)^{11}\mathrm{d}k = \frac{189\, \Gamma^{24}(\frac14)}{2^{15}\,65\,\pi^6}.
\end{equation}
Moments of higher powers of $K'$, the observation about rational constants above, as well as many other integrals involving the complete elliptic integrals, will be elaborated in a future paper (see footnote \ref{footnote1}).
\qede
\end{remark}

\section{Weight $3$ cases and lattice sums}\label{Sec: weight 3 cases}
In this section we note some additional formulas for critical $L$-values
of weight $3$ cusp forms.  The ideas for the proof below are borrowed from \cite{Rg}.
\begin{theorem} \label{weight3thm} 
Suppose that $f(\tau)=\eta^3(r\tau)\eta^3(s
\tau)$, where $r+s\equiv0 \ (\mathrm{mod} \ 8)$.  Then
\begin{equation}\label{weight 3 theta functions}
L(f,2)= \frac{8\pi^2}{\sqrt{rs^3}}\biggl(\sum_{n=0}^{\infty}(-1)^{\frac{n(n+1)}{2}}q^{\frac{(2n+1)^2}{8}}\biggr)^4 = \frac{8}{\sqrt{rs^3}} \, k_{r/s} \, k'_{r/s} \, K^2(k_{r/s}),
\end{equation}
where $q=e^{-\pi\sqrt{r/s}}$, and $k_p$ denotes the $p$th singular value of $K$.
\end{theorem}

\begin{proof}  The Jacobi triple product gives
\begin{equation*}
\eta^3(\tau)=\sum_{n=1}^{\infty}n\chi_{-4}(n)e^{\frac{2\pi i n^2
\tau}{8}}.
\end{equation*}
If we write $L(f,2)$ as a real-valued integral, then
\begin{equation*}
L(f,2)=4\pi^2\int_{0}^{\infty}\sum_{n,k\ge1}n k
\chi_{-4}(n)\chi_{-4}(k)e^{-\frac{2\pi r n^2 u}{8}-\frac{2\pi s k^2
u}{8}}u \d u.
\end{equation*}
Applying the involution for the weight $3/2$ theta function
leads to
\begin{equation*}
L(f,2)=\frac{4\pi^2}{\sqrt{s^3}}\int_{0}^{\infty}\sum_{n,k\ge1}n k
\chi_{-4}(n)\chi_{-4}(k)e^{-\frac{2\pi r n^2 u}{8}-\frac{2\pi k^2}{8s u}}\frac{\d u}{\sqrt{u}}.
\end{equation*}
By absolute convergence, we evaluate the integral first using standard results -- which produces a Bessel $K$ function with order $1/2$. Simplifying, we get
\begin{align*}
L(f,2)=&\frac{8\pi^2}{\sqrt{r s^3}}\sum_{n,k\ge 1}k
\chi_{-4}(k)\chi_{-4}(n)q^{\frac{n k}{2}} =\frac{8\pi^2}{\sqrt{r s^3}}\sum_{k\ge 1}\frac{k
\chi_{-4}(k)q^{\frac{k}{2}}}{1+q^k}\\
=&\frac{8\pi^2}{\sqrt{r
s^3}}\biggl(\sum_{k=0}^{\infty}(-1)^{\frac{k(k+1)}{2}}q^{\frac{(2k+1)^2}{8}}\biggr)^4.
\end{align*}
The final Lambert series identity follows from results in
Ramanujan's notebooks \cite{Be3}. The connection with singular values follows from  standard theta function manipulations, leading to
\[ L(f,2) = \frac{2\pi^2}{\sqrt{rs^3}} \theta_2^2(q)\theta_4^2(q), \]
and from the fact that $k_p = \theta_2^2(e^{-\pi \sqrt p})/\theta_3^2(e^{-\pi \sqrt p})$.
\end{proof}

\begin{theorem} The following evaluations are true: \label{thm table}

\begin{equation*}
{ \renewcommand{\arraystretch}{2.25}
    \begin{tabular}{|c|c|}
        \hline
	\vspace{-0.1cm} \boldmath{$f(\tau)$} &  \boldmath{$L(f,2)$} \vspace{-0.1cm}  \\  \hline \hline
      $\eta^6(4\tau)$ & $\displaystyle \frac{\Gamma^4(\frac14)}{64\pi}$\\
        $\eta^3(2\tau)\eta^3(6\tau)$ & $\displaystyle \frac{\Gamma^6(\frac13)}{2^{17/3} \pi^2}$ \\
        $\eta^3(\tau)\eta^3(7\tau)$ & $\displaystyle \frac{
        \Gamma^2(\frac17)\Gamma^2(\frac27)\Gamma^2(\frac47)}{224\pi^2}$\\
        $\eta^3(3\tau)\eta^3(5\tau)\pm\eta^3(\tau)\eta^3(15\tau)$ & $\displaystyle \frac{\Gamma(\frac{1}{15})\Gamma(\frac{2}{15})\Gamma(\frac{4}{15})\Gamma(\frac{8}{15})}{30\sqrt{54\mp 6}\,\pi}$ \\
        \hline
$\displaystyle \frac{\eta^5(4\tau)\eta^5(8\tau)}{\eta^2(2\tau)\eta^2(16\tau)}$ & $\displaystyle \frac{\Gamma^2(\frac18)\Gamma^2(\frac38)}{64\sqrt{2}\pi}$ \\ 
$\displaystyle \frac{\eta^{18}(8\tau)}{\eta^6(4\tau)\eta^6(16\tau)}$ & $\displaystyle \frac{\Gamma^4(\frac14)}{32 \sqrt{2} \pi}$ \\
$\eta^2(\tau)\eta(2\tau)\eta(4\tau)\eta^2(8\tau)$ & $\displaystyle \frac{\Gamma^2(\frac18) \Gamma^2(\frac38)}{192 \pi}$ \\
\hline
    \end{tabular}}
\end{equation*}
\end{theorem}
\begin{proof} For the first four entries in the table we use Theorem \ref{weight3thm}. The singular values $k_p$, as well as $K(k_p)$, are well tabulated in \cite{Be3} and \cite{Zu}. For instance, for the third entry we need 
\[ k_7 = \frac{\sqrt{2}(3-\sqrt 7)}{8}, \quad K(k_7) = \frac{\Gamma(\frac17)\Gamma(\frac27)\Gamma(\frac47)}{4\sqrt[4]{7}\pi}. \]
In general, $k_p$ is algebraic and $K(k_p)$ involves only algebraic numbers and Gamma functions. \\

The last three entries are weight 3 cusp forms in Martin's list \cite{Ma}. For the third last one, we convert it to the following integral,
\[ L(f,2)=\frac{1}{4\sqrt{2}}\int_{0}^{1}\frac{K(k)}{\sqrt{k(1-k)}}\d k, \]
which can be evaluated using the series representation of $K$ and interchanging the order of integration and summation. 
%This cusp form is lacunary (that is, the density of the nonzero coefficients is zero), 
%because it is a product of weight $3/2$ theta functions. 

For the second last entry, we have
\[L(f,2)=\frac{1}{8}\int_{0}^{1}\frac{K(k)}{\sqrt{k}(1-k^2)^{3/4}}\d k. \]
The evaluation of this integral follows from \cite[Eqn.~(8)]{wan}. Indeed, many integrals over a single $K$ have closed forms, and the two integrals we just evaluated can also be done by a computer algebra system.

For the last entry in the table, we obtain
\[ L(f,2)=\frac{1}{2\sqrt{2}}\int_{0}^{1}\frac{K(k)}{\sqrt{1+k}}\d k. \]
Denoting the integral by $I_1$, we use the moments of $K'$ found in \cite{wan} and a quadratic transformation to produce
\begin{equation}\label{i1int}
I_1 = \frac{1}{\sqrt2} \int_0^1 \frac{K'(x)}{\sqrt{1+x}}\mathrm dx = \frac{1}{2\sqrt2} \biggl[\frac{\Gamma^2(\frac18)\Gamma^2(\frac38)}{16\pi}-{}_4F_3\biggl({{\frac34,1,1,\frac54}\atop{\frac32,\frac32,\frac32}};1\biggr)\biggr].
\end{equation}
Similarly, with the auxiliary integral $I_2 :=\int_0^1 K(x)/\sqrt{x(1+x)}\, \mathrm dx$, we have
\begin{equation}\label{i2int}
I_2 = \frac{1}{\sqrt2} \int_0^1 \frac{K'(x)}{\sqrt{1-x}}\mathrm dx = \frac{1}{2\sqrt2} \biggl[\frac{\Gamma^2(\frac18)\Gamma^2(\frac38)}{16\pi}+{}_4F_3\biggl({{\frac34,1,1,\frac54}\atop{\frac32,\frac32,\frac32}};1\biggr)\biggr].
\end{equation}
Experimentally it is observed that $I_2=2I_1$, which can be shown as follows.

We know that $I_1/(2\sqrt2) = L(f,2)$. On the other hand, it is readily verifiable that $I_2/(2\sqrt2)=L(f_0,2)$, where $f_0(q):=-f(-q)$. Consequently, by looking at the $q$-expansion of $f_0-f$ and using the fact that the coefficients of $f$ are multiplicative, we deduce that $I_2-I_1=I_1$. The desired evaluation of $I_1$ now follows by combining \eqref{i1int} and \eqref{i2int}, and inter alia we also obtain a closed form for the $_4F_3$.
\end{proof}

\begin{remark}
Here we retain the notation used in the last part of the proof above. Firstly, with $f(\tau) = \eta^2(\tau)\eta(2\tau)\eta(4\tau)\eta^2(8\tau)$, it is true that
\[ f(\tau) = \frac12 \sum_{(n,m) \ne (0,0) } (m^2-2n^2)q^{m^2+2n^2}.\] 
Secondly, it is known that the series \eqref{L-series definition} for the $L$-value of a weight $k$ cusp form converges conditionally for some $s < (k+1)/2$. 
% \textcolor{red}{ref? Chandrasekaran} 
In particular, the series for $L(f,2)$ converges, and so by looking at the partial sums, we have
\[ I_1 = \sqrt{2} \, \sum_{(m, n) \ne (0,0)} \frac{m^2-2n^2}{(m^2+2n^2)^2}, \]
where we sum over expanding ellipses $m^2+2n^2 \le M, \, M \to \infty$. Similarly,
\[ I_2 = \sqrt{2} \, \sum_{(m, n) \ne (0,0)} (-1)^{m+1} \frac{m^2-2n^2}{(m^2+2n^2)^2}. \]
Subtracting the sums gives another proof that $I_2-I_1=I_1$. Since the sum for $I_2$ has better convergence properties (it can be summed over expanding rectangles), we will deal exclusively with alternating versions of the lattice sums we encounter. Note that we can decompose $f$ into weight $3/2$ theta functions $(\eta^2(\tau)\eta^2(4\tau)/\eta(2\tau)) \cdot (\eta^2(2\tau)\eta^2(8\tau)/\eta(4\tau))$ and therefore obtain a different series expansion. Such expansions imply the alternating sum identities
\begin{equation}
\sum_{(m, n) \ne (0,0)} \frac{(-1)^{m+1} (m^2-2n^2)}{(m^2+2n^2)^2} = \sum_{m,n} \frac{18(-1)^m(3m+1)(3n+1)}{((3m+1)^2+2(3n+1)^2)^2}= \frac{\Gamma^2(\frac18)\Gamma^2(\frac38)}{48\pi}.
\end{equation} \qede
\end{remark}

\begin{example}
The first entry in Theorem \ref{thm table} ($f(\tau) = \eta^6(4\tau)$) corresponds to the result 
\begin{equation} \label{table1st}  
L(f,2) = \frac14 \int_0^1 \frac{K(k)}{\sqrt{1-k^2}}\mathrm{d}k = \frac{\Gamma^4(\frac14)}{64\pi},
\end{equation}
and may be expressed as a lattice sum using a binary theta-expansion of $f$:
\begin{equation}\sum_{(m, n) \ne (0,0)} (-1)^{m+1}\frac{m^2-4n^2}{(m^2+4n^2)^2} = \frac{\Gamma^4(\frac14)}{32\pi}. \label{lattice4} \end{equation}
Combined with well-known lattice sum evaluations \cite{zucker}, consequences of \eqref{lattice4} include
\[ \sum_{(m, n) \ne (0,0)} \frac{(-1)^m n^2}{(m^2+4n^2)^2} = \frac{\Gamma^4(\frac14)}{256\pi}-\frac{3\pi\log2}{32}, \ \sum_{(m, n) \ne (0,0)} \frac{(-1)^{m+1} m^2}{(m^2+4n^2)^2} = \frac{\Gamma^4(\frac14)}{64\pi}+\frac{3\pi\log2}{8}.\]

On the other hand,
\[ L(f,1) = \frac{1}{\pi} \int_0^1 \frac{K(k)}{\sqrt{1-k^2}}\mathrm{d}k = \frac{\Gamma^4(\frac14)}{16\pi^2}, \]
which is consistent with the functional equation satisfied by $L(f,s)$. 
%Likewise, $L(f,1)$ for the last three entries in the table of Theorem \ref{thm table} can also be easily computed either from %the functional equations or from the integrals.

We note that \eqref{table1st} actually is a specialization of \cite[Eqn. after (5.14)]{duke}, which states
\[ \sum_{n=1}^\infty \frac{a_n}{n^2} q^{n/4} = \frac{\pi \sqrt k}{4K(k)}\,_3F_2\biggl({{\frac34,\frac34,1}\atop{\frac54,\frac54}};k^2\biggr), \]
where $a_n$ are the coefficients in the $q$-expansion of $f$, and as usual $k=\theta_2^2(q)/\theta_3^2(q)$; taking the limit $k \to 1^-$ and appealing to the Stolz-Ces\`aro theorem recovers \eqref{table1st}. \qede
\end{example}

\begin{example}
The second entry in Theorem \ref{thm table} ($f(\tau) = \eta^3(2\tau)\eta^3(6\tau)$) gives the non-trivial integral evaluation
\begin{equation} \label{Kcubict}
L(f,2) = \int_0^1 (3+6p)^{-\frac12} K\biggl(\frac{p^{\frac32}(2+p)^{\frac12}}{(1+2p)^{\frac12}}\biggr)\mathrm{d}p = \frac{\Gamma^6(\frac13)}{2^{\frac{17}3} \pi^2}, \end{equation}
where we have used the parametrization of the degree 3 modular equation and multiplier \cite[Ch.~19]{Be3}. We also have $L(f,1) = \sqrt{3}/\pi \, L(f,2)$, which can be shown from \eqref{Kcubict} either by the functional equation or by using a cubic transformation. We note that \eqref{Kcubict} (after a change of variable) appears in a very different context in \cite[Sec.~3]{walk2}.

The lattice sum associated with \eqref{Kcubict} is
\begin{equation} \sum_{(m, n) \ne (0,0)} (-1)^{m+n+1}\frac{m^2-3n^2}{(m^2+3n^2)^2} = \frac{\Gamma^6(\frac13)}{2^{\frac{14}3} \pi^2}. \end{equation}

For the third entry in Theorem \ref{thm table}, we do not seem to obtain a reasonably concise integral involving $K$ (due to the apparent lack of a parametrization for the degree 7 modular equation). Using the  binary theta function for $\eta^3(\tau)\eta^3(7\tau)$ \cite{chan}, we obtain the sum
\begin{equation}
\sum_{(m, n) \ne (0,0)} \frac{(-1)^{m}(2n^2-m^2)}{(m^2+mn+2n^2)^2} = \frac{\Gamma^2(\frac17)\Gamma^2(\frac27)\Gamma^2(\frac47)}{56\pi^2}.
\end{equation} \qede
\end{example}

\begin{example}
An $L$-value of the function $g$ used in the proof of \eqref{James formula} also gives some interesting lattice sum evaluations. Using the closed form for $L(g,3)$, the multiplicativity of the coefficients of $g$, and results in \cite{zucker}, we deduce
\begin{align} \nonumber
\sum_{(m, n) \ne (0,0)}\frac{(-1)^{m+n}m^2n^2}{(m^2+n^2)^3} & = \frac{\Gamma^8(\frac14)}{2^9 \,3\, \pi^3}-\frac{\pi \log 2}{8}, \\ 
\sum_{(m, n) \ne (0,0)}\frac{(-1)^{m+n}m^4}{(m^2+n^2)^3} & = -\frac{\Gamma^8(\frac14)}{2^9 \,3\, \pi^3}-\frac{3\pi \log 2}{8}, \\ \nonumber
\sum_{(m, n) \ne (0,0)}\frac{(-1)^{m}m^2n^2}{(m^2+n^2)^3} & = -\frac{\Gamma^8(\frac14)}{2^{10} \,3\, \pi^3}-\frac{\pi \log 2}{16}.
\end{align} \qede
\end{example}

\section{Even weight cases and conclusion}

For \textit{even weight} cusp forms we do not seem to obtain formulas for the $L$-values in terms of Gamma functions; instead hypergeometric functions are involved. Consider the following pair of weight $4$ examples:
\begin{equation*}
f_1(\tau)=\frac{\eta^{16}(4\tau)}{\eta^4(2\tau)\eta^4(8\tau)},\qquad \qquad f_2(\tau)=\eta^4(2\tau)\eta^4(4\tau).
\end{equation*}
All of the critical $L$-values of $f_1$ and $f_2$ reduce to special values of hypergeometric functions.  Furthermore, there are some very curious relations \textit{between} the $L$-values of both cusp forms:
\begin{equation}
\begin{split}
L(f_1,3)=\frac{\pi}{2}L(f_{2},2)= \frac{\pi^2}{8} L(f_1,1)&=\frac{1}{8}\int_{0}^{1}\frac{K(k)^2}{\sqrt{1-k^2}}\d k\\
 &= \frac{\pi^3}{32}\,{_4F_3}\biggl({{\frac12,\frac12,\frac12,\frac12}\atop{1,1,1}};1\biggr).
\end{split}
\end{equation}
The last equality follows from \cite[Eqn. 35]{wan}.  Similarly we have  
\begin{equation}
\begin{split}
L(f_2,3)=\frac{\pi}{4}L(f_1,2)=\frac{\pi^2}{4}L(f_2,1)= &  \frac18 \int_0^1 \frac{K(k)K'(k)}{\sqrt{1-k^2}}\d k=\frac14 \int_{0}^{1}K(k)^2 \d k \\
=&\frac{\pi^4}{128}\,{_7F_6}\biggl({{\frac54,\frac12,\frac12,\frac12,\frac12,\frac12,\frac12}\atop{\frac14,1,1,1,1,1}};1\biggr). \label{intk2}  
\end{split}
\end{equation}
The  second last equality comes from \cite[Sec.~5.4]{wan}, while the hypergeometric evaluation follows from \cite[Eqns.~3 and 18]{wan}.  Curiously this last integral also appears in connection with random walks \cite[Sec.~3]{walk2}.

\subsection{Conclusion}

As developed by Manin and Shimua (see e.\,g.~\cite[Thm.~1]{shimura2}), and also recorded in \cite[Sec.~3.4]{KZ},  the critical $L$-values of a modular form $f$ satisfy the following property: the ratios $L(f,2k)/L(f,2k-2)$ and $L(f,2k+1)/L(f,2k-1)$ can all be expressed in terms of algebraic (often rational) numbers and powers of $\pi$, where $s$ is the weight and $k=1,2,\ldots$. When the weight $s$ is odd, the functional equation relating $L(f,k)$ to $L(f,s-k)$ then implies that \textit{all} the critical $L$-values of $f$ are related by constants of said type. The explicit algebraic numbers involved may be found using Rankin's method, as explained in \cite{shimura1}, though the computation is highly non-trivial and tedious.
% moreover the product of any two elements, one from each sequence, is related to the Petersson norm. 

It was simple to directly verify this property for $f_1$ and $f_2$ studied above. On the other hand, more effort and ad~hoc strategies were needed to show that for the weight 5 cusp forms $g$ and $h$, all the critical $L$-values are related by rational multiples and powers of $\pi$. (Nevertheless, we believe these approaches are yet easier than Rankin's method for the cusp forms concerned.) It would be valuable to find a general and more approachable method for computing the ratio of two critical $L$-values.

We conclude with two more observations and directions for further research.

\begin{enumerate}
\item It seems that the critical $L$-values of some even weight cusp forms can be expressed as hypergeometric functions, while those of odd weight cusp forms can often be expressed in terms of Gamma functions. It would be interesting to explain this discrepancy, because the Ramanujan zeta function, and various other interesting zeta functions, are attached to such cusp forms (see e.\,g.~\cite{Rg2}). It would also be illuminating to see if  many other critical $L$-values of odd weight cusp forms evaluate in terms of Gamma functions.

\item As shown above, $L(f_1,s)/L(f_2,s-1)/\pi$ is a rational number. (Note that $f_2$ is $f_1$ twisted by the non-trivial Dirichlet character of conductor 4.) Are there other pairs of even weight modular forms with the same property? If so, is there a method to find, given one function in the pair, the other function?
\end{enumerate}

\medskip

\acknowledgements We would like to thank Wadim Zudilin for extremely helpful feedback, and for pointing out the references \cite{duke}, \cite{shimura1} and \cite{shimura2}. We also thank Yajun Zhou for showing us a proof of equation \eqref{old13}.

\end{document}